\documentclass[11pt]{amsart} 
\usepackage{amssymb,amsmath,latexsym,enumerate,graphicx}

\newtheorem{theorem}{Theorem}[section]
\newtheorem{corollary}[theorem]{Corollary}
\newtheorem{proposition}[theorem]{Proposition}
\newtheorem{lemma}[theorem]{Lemma}

\newtheorem*{definition}{Definition}

\newcommand\saw[1]{\left(\!\left( #1 \right)\!\right)}
\newcommand\fl[1]{\left\lfloor {#1} \right\rfloor} 
\newcommand\fr[1]{\left\{ {#1} \right\}}

\def\Z{\mathbb{Z}}
\def\Q{\mathbb{Q}}
\def\R{\mathbb{R}}
\def\C{\mathbb{C}}

\title{Finite Trigonometric Character Sums Via Discrete Fourier Analysis}

\author{Matthias Beck}
\author{Mary Halloran}
\address{Department of Mathematics\\
         San Francisco State University\\
         San Francisco, CA 94132\\
         U.S.A.}
\email{beck@math.sfsu.edu, mahall@sfsu.edu}

\thanks{The authors thank an anonymous referee for many helpful suggestions and corrections.}

\date{3 June 2008}

\begin{document}

\begin{abstract} 
We prove several old and new theorems about finite sums involving characters and trigonometric functions.
These sums can be traced back to theta function identities from Ramanujan's notebooks and were systematically first studied by Berndt and Zaharescu; their proofs involved complex contour integration. We show how to prove most of Berndt--Zaharescu's and some new identities by elementary methods of discrete Fourier Analysis.
\end{abstract}

\maketitle


\section{Introduction}

Curiously looking trigonometric identities such as
\begin{equation}\label{firstexample}
  \frac{ \sin^2 \left( \frac{ 3 \pi }{ 7 } \right) }{ \sin \frac{ 2 \pi }{ 7 } } - \frac{ \sin^2 \left( \frac{ 2 \pi }{ 7 } \right) }{ \sin \frac \pi 7 } + \frac{ \sin^2 \left( \frac \pi 7 \right) }{ \sin \frac{ 3 \pi }{ 7 } } = 0
\end{equation}
first surfaced in \cite{berndtzhang} as corollaries of (nontrivial) theta function identities from Ramanujan's notebooks, and also in \cite{liueisenstein}.
Berndt--Zaharescu initiated in \cite{berndtzaharescu} a systematic study of finite trigonometric identities involving characters, which nicely explained and vastly expanded on all previously known identities of the likes of \eqref{firstexample}. 
Let us give two examples of such identities, one due to Berndt--Zaharescu (which generalizes \eqref{firstexample}) \cite[Theorem 3.1]{berndtzaharescu} and one that we believe is novel.

\begin{theorem}[Berndt--Zaharescu] \label{finalthm}
Suppose $\chi$ is a nonprincipal, real, primitive, odd character modulo $k$, where $k \ge 7$ is odd. Then
\[
  \sum_{ j=1 }^{ k-1 } \chi(j) \frac{ \sin^2 \left( \frac{ \pi j }{ k } \right) }{ \sin \left( \frac{ 4 \pi j }{ k } \right)  }
  = \frac{ 3 \sqrt k }{ 2 } \left( \chi(2) - 1 \right) h(-k)
  = \begin{cases}
      -3 \sqrt k \, h(-k) & \text{ if } k \equiv 3 \bmod 8 , \\
      0 & \text{ if } k \equiv 7 \bmod 8 .
    \end{cases}
\]
\end{theorem}

\begin{theorem}\label{sinecharthm}
Let $\chi$ be a nonprincipal, real, odd character modulo $k$, let $a$ and $b$ be positive integers such that $a$ is odd and $1 \le ab < k$, and let $x \in \R$. Then
\[
  \sum_{ j=0 }^{ k-1 } \chi(j) \sin^a \left( \frac{ 2 \pi bj }{ k } + x \right) = \frac{ \sqrt k }{ 2^{ a-1} } \sum_{ { n, m \ge 0 } \atop { n+2mb = ab } } (-1)^{ m-(a-1)/2 } \binom a m \cos \left( (a-2m) x \right) \chi(n) \, .
\]
\end{theorem}

\noindent
Here $h(-k)$ denotes the \emph{class number} (i.e., the number of equivalent classes of fractional ideals) of the imaginary quadratic field $\Q[\sqrt{-k}]$.
Class numbers make an appearance in many of the theorems that will follow, and so one can think of them as representations of class numbers by trigonometric sums. 

Examining the ingredients of the left-hand sides of Theorems \ref{finalthm} and \ref{sinecharthm} suggests that there ought to be a way of interpreting them as convolutions of finite Fourier series, and this is our goal in this paper: to show that the elementary methods of discrete Fourier Analysis are well suited to prove these and many other theorems on finite sums involving characters and trigonometric functions. (Berndt--Zaharescu used complex contour integration to prove their theorems.)

We will give a quick brush-up on discrete Fourier Analysis and the Fourier transforms of the basic trigonometric functions in Section~\ref{convolutionsection}. In Section~\ref{easytrigsumsection} we show that this setting gives rise to essentially trivial proofs of many old and some new trigonometric identities not involving characters. In Section~\ref{charactersection} we give Fourier proofs of several old and new theorems along the lines of (and including) Theorems \ref{finalthm} and~\ref{sinecharthm}.


\section{Discrete Fourier Transforms and Convolutions}\label{convolutionsection}

We recall a few basic facts from discrete Fourier analysis (see, e.g., \cite{terrasfinitefourier}). Throughout the paper, we fix a positive integer $k$ and let $\omega = e^{ 2 \pi i/k }$.

\begin{theorem}\label{dftransthm}
Suppose $f$ is a periodic function on $\Z$ with period $k$. Then there exist $\hat f (0), \hat f (1), \dots, \hat f (k-1)$ such that
\[
  f(n) = \frac 1 k \sum_{ j=0 }^{ k-1 } \hat f (j) \, \omega^{ jn } .
\]
We can think of the discrete Fourier transform $\hat f$ as another periodic function on $\Z$ with period $k$.
Explicitly, $\hat f$ is given through
\[
  \hat f (n) = \sum_{ j=0 }^{ k-1 } f(j) \, \omega^{ -jn } .
\]
\end{theorem}

\begin{corollary}
$f$ is even/odd if and only if $\hat f$ is even/odd, respectively.
\end{corollary}

\begin{theorem}\label{convolutionthm}
Let $f$ and $g$ be periodic functions on $\Z$ with period $k$, and suppose $\hat f$ and $\hat g$ are their discrete Fourier transforms.
Then the \emph{convolution} $f * g$ of $f$ and $g$,
\[
  (f*g) (n) := \sum_{ j=0 }^{ k-1 } f(j) \, g(n-j) \, ,
\]
has period $k$ and discrete Fourier transform
\[
  \widehat{ f*g } = \hat f \cdot \hat g \, .
\]
\end{theorem}

\begin{corollary}\label{convolutioncor}
Let $f$ and $g$ be periodic functions on $\Z$ with period $k$, and suppose $\hat f$ and $\hat g$ are their discrete Fourier transforms. Then
\[
  \sum_{ m=0 }^{ k-1 } f(m) g(-m) = \frac 1 k \sum_{ j=0 }^{ k-1 } \hat f (j) \, \hat g (j) \, .
\]
\end{corollary}

\begin{corollary}\label{convolutionoddeven}
Let $f$ and $g$ be periodic functions on $\Z$ with period $k$, and suppose $\hat f$ and $\hat g$ are their discrete Fourier transforms.
If $f$ or $g$ is odd then
\[
  \sum_{ m=0 }^{ k-1 } f(m) \, g(m) = - \frac 1 k \sum_{ j=0 }^{ k-1 } \hat f (j) \, \hat g (j) \, .
\]
If $f$ or $g$ is even then
\[
  \sum_{ m=0 }^{ k-1 } f(m) \, g(m) = \frac 1 k \sum_{ j=0 }^{ k-1 } \hat f (j) \, \hat g (j) \, .
\]
\end{corollary}

As examples of discrete Fourier transforms, we discuss the basic trigonometric functions. It is through their discrete Fourier transforms that we will derive the finite trigonometric identities that we are after.
Let $\fr x = x - \fl x$ denote the \emph{fractional part} of $x$.

\begin{lemma}\label{trigtransforms}
Let $a \in \Z$, $0 < a < k$, $a \ne \frac k 2$.
If
\[
  f(n) =
  \begin{cases}
    \frac 1 2 & \text{ if } k | a+n , \\
    - \frac 1 2 & \text{ if } k | a-n , \\
    0 & \text{ otherwise, }
  \end{cases}
  \qquad \text{ then } \qquad
  \hat f (n) = i \sin \frac{ 2 \pi a n }{ k } \, .
\]
If
\[
  f(n) =
  \begin{cases}
    \frac 1 2 & \text{ if } k | a+n \text{ or } k | a-n , \\
    0 & \text{ otherwise, }
  \end{cases}
  \qquad \text{ then } \qquad
  \hat f (n) = \cos \frac{ 2 \pi a n }{ k }\, .
\]
If $k$ is odd and
\[
  f(n) = \saw{ \frac n k } :=
  \begin{cases}
    0 & \text{ if } k | n , \\
    \fr {\frac n k} - \frac 1 2 & \text{ otherwise, } 
  \end{cases}
  \qquad \text{ then } \qquad
  \hat f(n) =
  \begin{cases}
    0 & \text{ if } k | n , \\
    \frac i 2 \cot \frac{ \pi n }{ k } & \text{ otherwise. } 
  \end{cases}
\]
If $k$ is odd and
\[
  f(n) =
  \begin{cases}
    0 & \text{ if } k | n , \\
    (-1)^{ n \bmod k }  & \text{ otherwise, } 
  \end{cases}
  \qquad \text{ then } \qquad
  \hat f (n) = i \tan \frac{ \pi n }{ k } \, .
\]
Here $n \bmod k$ denotes the least positive residue mod k.
\end{lemma}


\section{Explicit Evaluations of Finite Trigonometric Sums}\label{easytrigsumsection}

By using the results of the previous section, we can obtain numerous identities of finite trigonometric sums, in an essentially trivial way. As an archetype, we prove the following identity, which is the earliest trigonometric sum evaluation that we are aware of, found by Stern in 1861 \cite{stern}.

\begin{proposition}\label{tangentsquared}
If $k$ is an odd positive integer,
\[
  \sum_{ j=1 }^{ k-1 } \tan^2 \left( \frac{ \pi j }{ k } \right) = k^2 - k \, .
\]
\end{proposition}

We refer to the excellent bibliography of \cite{berndtyeap} for further past appearances of this and all other trigonometric sums in this section.

\begin{proof}
The discrete Fourier transform of 
\[
  f(n) =
  \begin{cases}
    0 & \text{ if } k | n , \\
    (-1)^{ n \bmod k }  & \text{ otherwise, } 
  \end{cases}
\]
is $\hat f(n) = i \tan \frac{ \pi n }{ k }$; hence Corollary \ref{convolutionoddeven} gives
\[
  - \frac 1 k \sum_{ j=1 }^{ k-1 } \tan^2 \left( \frac{ \pi j }{ k } \right) = - \sum_{ j=0 }^{ k-1 } f(j)^2 = - \sum_{ j=1 }^{ k-1 } (-1)^{ 2n } = - (k-1) \, . \qedhere
\]
\end{proof}

Similar identities, which can be proved in a completely analogous fashion, include
\[
  \sum_{ j=1 }^{ k-1 } \cot^2 \left( \frac{ \pi j }{ k } \right) = \frac 1 3 (k-1) (k-2)
\]
and, for odd $k$,
\[
  \sum_{ j=1 }^{ k-1 } \csc^2 \left( \frac{ 2 \pi j }{ k } \right) = \frac 1 3 \left( k^2 - 1 \right) 
  \qquad \text{ and } \qquad
  \sum_{ j=1 }^{ k-1 } \sec^2 \left( \frac{ 2 \pi j }{ k } \right) = k^2 - 1 \, .
\]
One can easily establish these identities with the help of Corollary \ref{convolutionoddeven} and the identities
\[ \csc (2 \theta) = \tfrac 1 2 \left( \tan \theta + \cot \theta \right) \qquad \text{ and } \qquad \sec^2 \theta = 1 + \tan^2 \theta \, . \]
Another application of the Convolution Theorem is the following identity proved by Eisenstein \cite{eisenstein}; the earliest published proof we are aware of is in \cite{stern}.

\begin{proposition}\label{eisenstein}
For $0 < a < k$,
\[
  \sum_{ j=1 }^{ k-1 } \cot \frac{ \pi j }{ k } \sin \frac{ 2 \pi a j }{ k } = k - 2a \, .
\]
\end{proposition}

\begin{proof}
If $a = \frac k 2$ then both sides of the identity are zero. Now suppose $a \ne \frac k 2$.
Let $f(n) = \saw{ \frac n k }$ and
\[
  g(n) =
  \begin{cases}
    \frac 1 2 & \text{ if } k | a+n , \\
    - \frac 1 2 & \text{ if } k | a-n , \\
    0 & \text{ otherwise. }
  \end{cases}
\]
Then
\[
  \hat f(n) =
  \begin{cases}
    0 & \text{ if } k | n , \\
    \frac i 2 \cot \frac{ \pi n }{ k } & \text{ otherwise} 
  \end{cases}
  \qquad \text{ and } \qquad
  \hat g (n) = i \sin \frac{ 2 \pi a n }{ k } \, ,
\]
and by Corollary \ref{convolutionoddeven},
\[
  - \frac{ 1 }{ 2k } \sum_{ j=1 }^{ k-1 } \cot \frac{ \pi j }{ k } \sin \frac{ 2 \pi a j }{ k } 
  = - \sum_{ j=0 }^{ k-1 } f(j) \, g(j)
  = - \left( \frac a k - \frac 1 2 \right) \left( - \frac 1 2 \right) - \left( \frac{ k-a }{ k } - \frac 1 2 \right) \left( \frac 1 2 \right) 
  = - \frac{ k - 2a }{ 2k } \, . \qedhere
\]
\end{proof}

Generalizations of Proposition \ref{eisenstein} were constructed by Williams and Zhang \cite{williamszhang}, who allowed arbitrary powers of the cotangent, and Wang \cite{wangtrig}, who replaced the cotangent by a cotangent derivative.

Here we give siblings of Proposition \ref{eisenstein} that have essentially identical proofs in the language of discrete Fourier analysis.

\begin{proposition}
Suppose $k$ is an odd, positive integer and $0<a<k$. Then
\begin{align*}
  \sum_{ j=1 }^{ k-1 } \tan \frac{ \pi j }{ k } \sin \frac{ 2 \pi a j }{ k } &= (-1)^{ a+1 } k \, , \\
  \sum_{ j=1 }^{ k-1 } \sin \frac{ 2 \pi a j }{ k } \csc \frac{ 2 \pi j }{ k } &= 
    \begin{cases}
      k-a & \text{ if $a$ is odd, } \\
      -a & \text{ if $a$ is even, }
    \end{cases} \\
  \sum_{ j=1 }^{ k-1 } \tan \frac{ \pi j }{ k } \csc \frac{ 2 \pi j }{ k } &= \frac{ k^2 - 1 }{ 2 } \, , \\
  \sum_{ j=1 }^{ k-1 } \cot \frac{ \pi j }{ k } \csc \frac{ 2 \pi j }{ k } &= \frac{ k^2 - 1 }{ 6 } \, .
\end{align*}
\end{proposition}
These examples illustrate that the method of using discrete Fourier convolution does not discriminate between different trigonometric functions. On the other hand, a disadvantage compared to other methods is the fact that we cannot easily handle arbitrary powers, which appear, e.g., in identities of Berndt--Yeap \cite{berndtyeap} and Chu--Marini \cite{chumarini}. We can, however, deal with fixed higher powers. As an illustrating example, we prove the following.

\begin{proposition}
If $k$ is an odd positive integer,
\[
  \sum_{ j=0 }^{ k-1 } \tan^4 \left( \frac{ \pi j }{ k } \right) = \frac 1 3 k (k-1) \left( k^2 + k - 3 \right) .
\]
\end{proposition}

\begin{proof}
We will use the Convolution Theorem with the function
\[
  g(n) = \frac 1 k \sum_{ j=0 }^{ k-1 } \tan^2 \frac{ \pi j }{ k } \omega^{ jn } .
\]
This function, in turn, is a convolution, namely $g(n) = - \sum_{ m=0 }^{ k-1 } f(n-m) \, f(m)$ where
\[
  f(n) =
  \begin{cases}
    0 & \text{ if } k | n , \\
    (-1)^{ n \bmod k }  & \text{ otherwise. } 
  \end{cases}
\]
By Proposition~\ref{tangentsquared}, $g(0) = k-1$.
Now for $0 < n < k$ we have with $f(0) = 0$
\[
  g(n) = - \sum_{ m=0 }^{ n-1 } f(n-m) \, f(m) - \sum_{ m=n+1 }^{ k-1 } f(n-m) \, f(m) \, .
\]
We have $f(m) = (-1)^m$ for $0<m<k$, but the evaluation of $f(n-m)$ is slightly more subtle:
\[
  f(n-m) = 
  \begin{cases}
  (-1)^{ n-m } & \text{ if } 0 < m < n , \\
  (-1)^{ n-m+1 } & \text{ if } n < m < k .
  \end{cases}
\]
Here the first case is clear since $0<m<n$ implies $0 < n-m < k$.
The second case $n<m<k$ implies $-k < n-m < 0$, and we use the fact that $f$ is an odd function (since tangent is odd).
Thus we obtain for $0 < n < k$
\[
  g(n) 
  = - \sum_{ m=1 }^{ n-1 } (-1)^{ n-m } (-1)^m - \sum_{ m=n+1 }^{ k-1 } (-1)^{ n-m+1 } (-1)^m
  = (-1)^n (k-2n) \, .
\]
In summary, we have
\[
  g(n) =
  \begin{cases}
  k-1 & \text{ if } n|k , \\
  (-1)^{ n \bmod k } \left( k - 2 (n \bmod k) \right) & \text{ otherwise. }
  \end{cases}
\]
Now we can apply Corollary~\ref{convolutionoddeven}:
\[
  \frac 1 k \sum_{ j=0 }^{ k-1 } \tan^4 \left( \frac{ \pi j }{ k } \right)
  = \sum_{ m=0 }^{ k-1 } \left( g(m) \right)^2
  = (k-1)^2 + \sum_{ m=1 }^{ k-1 } (2m-k)^2
  = \frac 1 3 (k-1) \left( k^2 + k - 3 \right) . \qedhere
\]
\end{proof}

As a final application of discrete Fourier convolution in this section, we derive two identities of binomial coefficients using Fourier analysis.
In preparation, we prove an identity for arbitrary sine and cosine powers (for which we do not need to apply Fourier convolution).

\begin{lemma}\label{sinelemma}
Let $a, b$ be positive integers such that $0 < ab < k$, and let $x, y \in \R$. Then
\[
  \sum_{ j=0 }^{ k-1 } \sin^{ 2a } \left( \frac{ b \pi j }{ k } + x \right) 
  = \sum_{ j=0 }^{ k-1 } \cos^{ 2a } \left( \frac{ b \pi j }{ k } + y \right)
  = \frac{ k }{ 2^{ 2a } } \binom{ 2a }{ a } .
\]
\end{lemma}

\begin{proof}
The cosine identity follows from the sine identity by replacing $x$ with $x + \frac \pi 2$.
The sine identity is obtained through
\begin{align*}
  \sum_{ j=0 }^{ k-1 } \sin^{ 2a } \left( \frac{ b \pi j }{ k } + x \right) \omega^{ nj }
  &= \frac{ 1 }{ (2i)^{ 2a } } \sum_{ j=0 }^{ k-1 } \left( \omega^{ bj/2 } e^{ ix } - \omega^{ - bj/2 } e^{ -ix } \right)^{ 2a } \omega^{ nj } \\
  &= \frac{ (-1)^a }{ 2^{ 2a } } \sum_{ j=0 }^{ k-1 } \sum_{ m=0 }^{ 2a } \binom{ 2a }{ m } \omega^{ (2a-m) bj/2 } e^{ (2a-m) ix } (-1)^m \omega^{ -mbj/2 } e^{ -mix } \omega^{ nj } \\
  &= \frac{ (-1)^a }{ 2^{ 2a } } \sum_{ m=0 }^{ 2a } (-1)^m \binom{ 2a }{ m } e^{ 2ix (a-m) } \sum_{ j=0 }^{ k-1 } \omega^{ bj (a-m) + nj } .
\end{align*}
Now let $n=0$. We have
\[
  \sum_{ j=0 }^{ k-1 } \omega^{ bj (a-m) } =
  \begin{cases}
    k & \text{ if } k | b(a-m) , \\
    0 & \text{ otherwise. }
  \end{cases}
\]
Since $0<ab<k$, $k$ divides $b(a-m)$ only if $m=a$, whence the above sum becomes
\[
  \sum_{ j=0 }^{ k-1 } \sin^{ 2a } \left( \frac{ b \pi j }{ k } + x \right) = \frac{ k }{ 2^{ 2a } } \binom{ 2a }{ a } \, . \qedhere
\]
\end{proof}

\begin{theorem}
For any positive integer $a$,
\[
  \sum_{ m=0 }^{ a } {\binom a m}^2 
  = \binom{ 2a }{ a }
  = \sum_{ m=0 }^{ a } (-1)^m \, 2^{ 2(a-m) } \binom{ 2m }{ m } \binom a m \, .
\]
\end{theorem}

\begin{proof}
Given $a \ge 1$, choose $k$ such that $2a < k$.
A computation very similar to the previous proof (with $b=2$ and $x=0$) gives
\[
  f(n)
  = \frac 1 k \sum_{ j=0 }^{ k-1 } \left( i \sin \frac{ 2 \pi j }{ k } \right)^a \omega^{ jn }
  = \frac{ 1 }{ 2^a } \sum_{ { 0 \le m \le a } \atop { k|n+a-2m } } (-1)^m \binom{ a }{ m } \, .
\]
Note that $f(-n) = (-1)^a f(n)$ and $- \frac k 2 < a-2m < \frac k 2$. Thus given any $n$, there is at most one $m$ such that $k|n+a-2m$. Therefore,
\[
  f(n)
  = \begin{cases}
    \frac{ (-1)^m }{ 2^a } \binom{ a }{ m } & \text{ if $k|n+a-2m$ for some $0 \le m \le a$,} \\
    0 & \text{ otherwise.}
  \end{cases}
\]
We claim that we can find $n$ (between 0 and $k-1$) such that $k|n+a-2m$, for any $0 \le m \le a$. Indeed, if $m < \frac a 2$ then choose $n = k - (a-2m))$, and if $m \ge \frac a 2$ then choose $n = -(a-2m)$. Thus by Lemma~\ref{sinelemma} and the Convolution Theorem,
\[
  \frac{ (-1)^a \binom{ 2a }{ a } }{ 2^{ 2a } } 
  = \frac{ (-1)^a }{ k } \sum_{ j=0 }^{ k-1 } \sin^{ 2a } \left( \frac{ 2 \pi j }{ k } \right) 
  = \sum_{ n=0 }^{ k-1 } f(-n) \, f(n)
  = \frac{ (-1)^a }{ 2^{ 2a } } \sum_{ m=0 }^a {\binom a m}^2 .
\]

The second identity follows with 
\[
  \cos^{ 2a } \left( \frac{ 2 \pi j }{ k } \right) 
  = \left( 1 - \sin^2 \left( \frac{ 2 \pi j }{ k } \right) \right)^a
  = \sum_{ m=0 }^a (-1)^m \binom a m \sin^{ 2m } \left( \frac{ 2 \pi j }{ k } \right) ,
\]
Lemma~\ref{sinelemma}, and the first identity:
\[
  \frac{ k \binom{ 2a }{ a } }{ 2^{ 2a } }
  = \sum_{ j=0 }^{ k-1 } \cos^{ 2a } \left( \frac{ 2 \pi j }{ k } \right)
  = \sum_{ m=0 }^a (-1)^m \binom a m \frac{ k \binom{ 2m }{ m } }{ 2^{ 2m } } \, . \qedhere
\]
\end{proof}


\section{Evaluations of Finite Character Sums}\label{charactersection}

We review a few definitions on characters.
A \emph{character} $\chi$ (mod $k$) is a nonzero map from $\Z_k$ to $\C$ such that $\chi(ab) = \chi(a) \chi(b)$ and $\chi(n) = 0$ when $\gcd (n,k) > 1$.
The \emph{principal character} mod $k$ is 
\[
  \chi_0(n) = \begin{cases}
  1 & \text{ if } \gcd(n,k) = 1, \\
  0 & \text{ otherwise. }
  \end{cases}
\]
One of the most basic properties of nonprincipal characters $\chi$ is the fact that
$
  \sum_{ j=0 }^{ k-1 } \chi(j) = 0
$.
A character $\chi$ (mod $k$) is \emph{induced} by a character $\chi'$ (mod $k'$) if $k'|k$ and $\chi = \chi'$ (where we extend $\chi'$ naturally to $\Z_k$).
A character that is not induced by any other character is \emph{primitive}.
We will use the fact that for a nonprincipal, real, primitive character modulo $k$, where $k$ is odd, 
\[
  \chi(2) = \begin{cases}
  1 & \text{ if } k \equiv \pm 1 \bmod 8 , \\
  -1 & \text{ if } k \equiv \pm 3 \bmod 8 .
  \end{cases}
\]
The reason that class numbers show up in our formulas is the following classic result (see, e.g., \cite[p.~344--346]{borevichshafarevich}):

\begin{theorem}\label{classnumthm}
Let $\chi$ be a nonprincipal, real, primitive, odd character modulo $k$, where $k \ge 7$. Then
\[
  h(-k) 
  = - \frac 1 k \sum_{ j=1 }^{ k-1 } j \, \chi(j)
  = \frac{ 1 }{ 2 - \chi(2) } \sum_{ j=1 }^{ (k-1)/2 } \chi(j) \, .
\]
\end{theorem}

The \emph{Gauss sum} $G(z, \chi)$ is defined for any $z \in \C$ through $G(z,\chi) = \sum_{ j=0 }^{ k-1 } \chi(j) \, \omega^{ jz }$.
Some basic results about Gauss sums are summarized next (see, e.g., \cite{berndtevanswilliams}).

\begin{theorem}
For each integer $n$, $\chi$ is real and primitive if and only if $G(n,\chi) = \chi(n) \, G(1,\chi) =: \chi(n) \, G(\chi)$.
Now let $\chi$ be a real, primitive character modulo $k$. Then
\[
  G(\chi) =
  \begin{cases}
    \sqrt k & \text{ if $\chi$ is even, } \\
    i \sqrt k & \text{ if $\chi$ is odd. }
  \end{cases}
\]
Consequently, for an integer $n$,
\[
  \frac 1 k \sum_{ j=0 }^{ k-1 } \chi(j) \, \omega^{ jn }
  = \frac 1 k G(n,\chi)
  = \begin{cases}
  \frac{ 1 }{ \sqrt k } \, \chi(n) & \text{ if $\chi$ is even, } \\
  \frac{ i }{ \sqrt k } \, \chi(n) & \text{ if $\chi$ is odd. }
  \end{cases}
\]
In other words,
\[
  \widehat \chi (n) = \begin{cases}
  \sqrt k \, \chi(n) & \text{ if $\chi$ is even, } \\
  - i \sqrt k \, \chi(n) & \text{ if $\chi$ is odd. }
  \end{cases}
\]
\end{theorem}

Now we're ready to apply the machinery of Section \ref{convolutionsection} to prove old and new theorems about finite trigonometric sums involving characters. We start with proving Theorem \ref{sinecharthm}, which says that, for $\chi$ a nonprincipal, real, odd character modulo $k$, $a$ and $b$ positive integers such that $a$ is odd and $1 \le ab < k$, and $x \in \R$, we have
\[
  \sum_{ j=0 }^{ k-1 } \chi(j) \sin^a \left( \frac{ 2 \pi bj }{ k } + x \right) = \frac{ \sqrt k }{ 2^{ a-1} } \sum_{ { n, m \ge 0 } \atop { n+2mb = ab } } (-1)^{ m-(a-1)/2 } \binom a m \cos \left( (a-2m) x \right) \chi(n) \, .
\]

\begin{proof}[Proof of Theorem \ref{sinecharthm}]
Let $f(n) = \frac 1 k \sum_{ j=0 }^{ k-1 } \left( i \sin \left( \frac{ 2 b \pi j }{ k } + x \right) \right)^a \omega^{ nj }$; then
\[
  f(n) 
  = \frac{ 1 }{ 2^a k } \sum_{ m=0 }^{ a } (-1)^m \binom a m e^{ (a-2m) xi } \sum_{ j=0 }^{ k-1 } \omega^{ (n+(a-2m)b)j }
  = \frac{ 1 }{ 2^a } \sum_{ { 0 \le m \le a } \atop { k|n+(a-2m)b } } (-1)^m \binom a m e^{ (a-2m) xi } .
\]
If we let $s_m = (a-2m)b$, then $s_{ a-m } = - s_m$, and we can rewrite the above equation as
\[
  f(n) = \frac{ 1 }{ 2^a } \left( \sum_{ { 0 \le m < a/2 } \atop { k|n+s_m } } (-1)^m \binom a m e^{ (a-2m)xi } - \sum_{ { 0 \le m < a/2 } \atop { k|n-s_m } } (-1)^m \binom a m e^{ -(a-2m)xi } \right) .
\]
Let $g(n) = \frac{ i }{ \sqrt k } \, \chi(n)$, then by Corollary \ref{convolutionoddeven},
\begin{align*}
  &\frac{ i^a }{ k } \sum_{ j=0 }^{ k-1 } \chi(j) \sin^a \left( \frac{ 2 \pi b j }{ k } + x \right) \\
  &\qquad= - \frac{ i }{ 2^a \sqrt k } \sum_{ n=0 }^{ k-1 } \chi(n) \left( \sum_{ { 0 \le m < a/2 } \atop { k|n+s_m } } (-1)^m \binom a m e^{ (a-2m)xi } - \sum_{ { 0 \le m < a/2 } \atop { k|n-s_m } } (-1)^m \binom a m e^{ -(a-2m)xi } \right) \\
  &\qquad= - \frac{ i }{ 2^a \sqrt k } \sum_{ 0 \le m < a/2 } (-1)^m \binom a m \left( \sum_{ { 0 \le n < k } \atop { k|n+s_m } } \chi(n) e^{ (a-2m)xi } - \sum_{ { 0 \le n < k } \atop { k|n-s_m } } \chi(n) e^{ -(a-2m)xi } \right) .
\end{align*}
Since $1 \le ab \le k$ and $0 \le m < \frac a 2$, we have $0 < s_m \le ab < k$. Furthermore, since $0 \le n < k$, we have for all $m$
\[
  0 < n + s_m < 2k \qquad \text{ and } \qquad -k < n - s_m < k \, .
\]
Thus $k|n+s_m$ only when $n = k-s_m$ and $k|n-s_m$ only when $n = s_m$. Since for every $m$ there exists exactly one $n$ such that $n = s_m$ and $\chi(k-s_m) = - \chi(s_m)$, then
\begin{align*}
  \frac{ i^a }{ k } \sum_{ j=0 }^{ k-1 } \chi(j) \sin^a \left( \frac{ 2 \pi b j }{ k } + x \right)
  &= - \frac{ i }{ 2^a \sqrt k } \sum_{ 0 \le m < a/2 } (-1)^m \binom a m \left( \chi(k-s_m) e^{ (a-2m)xi } - \chi(s_m) e^{ (a-2m)xi } \right) \\
  &= \frac{ i }{ 2^a \sqrt k } \sum_{ 0 \le m < a/2 } (-1)^m \binom a m \chi(s_m) \left( e^{ (a-2m)xi } + e^{ -(a-2m)xi } \right) \\
  &= \frac{ i }{ 2^{a-1} \sqrt k } \sum_{ { n, m \ge 0 } \atop { n+2mb = ab } } (-1)^m \binom a m \cos \left( (a-2m) x \right) \chi(n) \, . \qedhere
\end{align*}
\end{proof}

\begin{corollary}
Let $\chi$ be a nonprincipal, real, odd character modulo $k$ and let $b$ be a positive integer not divisible by $k$. Then
\[
  \sum_{ j=0 }^{ k-1 } \chi(j) \sin \frac{ 2 \pi b j }{ k } = \sqrt k \, \chi(b) \, .
\]
\end{corollary}

We can use essentially the same proof (or the substitution $x \mapsto x + \frac \pi 2$) to arrive at the identity
\[
  \sum_{ j=0 }^{ k-1 } \chi(j) \cos^a \left( \frac{ 2 \pi bj }{ k } + x \right) = \frac{ \sqrt k }{ 2^{ a-1} } \sum_{ { n, m \ge 0 } \atop { n+2mb = ab } } \binom a m \sin \left( (a-2m) x \right) \chi(n) \, ,
\]
which holds under the same conditions as in Theorem \ref{sinecharthm}.

The next result we will tackle appeared in \cite[Corollary 2.3]{berndtzaharescu}.

\begin{theorem}[Berndt--Zaharescu]\label{bzcot}
Suppose $\chi$ is a nonprincipal, real, primitive, odd character modulo $k$, where $k \ge 7$ is odd. Then
\[
  \sum_{ j=1 }^{ k-1 } \chi(j) \cot \frac{ \pi j }{ k } = 2 \sqrt k \, h(-k) \, .
\]
\end{theorem}

\begin{proof} 
We apply Corollary \ref{convolutionoddeven} with the functions $f(n) = \saw{ \frac n k }$ and $g(n) = \frac{ i }{ \sqrt k } \chi (n)$:
\[
  \frac{ i }{ 2k } \sum_{ j=1 }^{ k-1 } \chi(j) \cot \frac{ \pi j }{ k }
  = - \frac{ i }{ \sqrt k } \sum_{ m=1 }^{ k-1 } \left( \frac m k - \frac 1 2 \right) \chi(m)
  = \frac{ i }{ \sqrt k } \, h(-k) \, ,
\]
by Theorem \ref{classnumthm}.
\end{proof}

The following extension of Theorem \ref{bzcot} seems to have gone unnoticed in \cite{berndtzaharescu}.

\begin{corollary}
Suppose $\chi$ is a nonprincipal, real, primitive, odd character modulo $k$, where $k \ge 7$ is odd, and $\gcd(b,k) = 1$. Then
\[
  \sum_{ j=1 }^{ k-1 } \chi(j) \cot \frac{ \pi jb }{ k } = 2 \sqrt k \, \chi(b) \, h(-k) \, .
\]
\end{corollary}

\begin{proof}
Since $\gcd(b,k) = 1$, we have
\[
  2 \sqrt k h(-k) = \sum_{ j=1 }^{ k-1 } \chi(j) \cot \frac{ \pi j }{ k } = \sum_{ j=1 }^{ k-1 } \chi(jb) \cot \frac{ \pi jb }{ k } = \chi(b) \sum_{ j=1 }^{ k-1 } \chi(j) \cot \frac{ \pi jb }{ k } \, .
\]
The statement now follows, since $\chi(b) = \pm 1$.
\end{proof}

With an essentially identical proof to that of Theorem \ref{bzcot} we obtain the tangent analog \cite[Corollary 5.2]{berndtzaharescu}.

\begin{theorem}[Berndt--Zaharescu]\label{bztan}
Suppose $\chi$ is a nonprincipal, real, primitive, odd character modulo $k$, where $k \ge 7$ is odd. Then
\[
  \sum_{ j=1 }^{ k-1 } \chi(j) \tan \frac{ \pi j }{ k } 
  = \sqrt k \left( 2 - 4 \chi(2) \right) h(-k)
  = \begin{cases}
  6 \sqrt k \, h(-k) & \text{ if } k \equiv 3 \bmod 8 , \\
  -2 \sqrt k \, h(-k) & \text{ if } k \equiv 7 \bmod 8 .
  \end{cases}
\]
\end{theorem}

Theorems \ref{bzcot} and \ref{bztan} and the identity $\csc 2 \theta = \frac 1 2 \left( \tan \theta + \cot \theta \right)$ immediately yield the following result.

\begin{corollary}
Suppose $\chi$ is a nonprincipal, real, primitive, odd character modulo $k$, where $k \ge 7$ is odd. Then
\[
  \sum_{ j=1 }^{ k-1 } \chi(j) \csc \frac{ 2 \pi j }{ k } 
  = 2 \sqrt k \left( 1 - \chi(2) \right) h(-k)
  = \begin{cases}
  4 \sqrt k \, h(-k) & \text{ if } k \equiv 3 \bmod 8 , \\
  0 & \text{ if } k \equiv 7 \bmod 8 .
  \end{cases}
\]
\end{corollary}

At this point we have all the Fourier series ingredients to prove the following theorem \cite[Corollary 2.2]{berndtzaharescu}; however, its Fourier proof is lengthy, so that we omit the details.

\begin{theorem}[Berndt--Zaharescu]
Suppose $\chi$ is a nonprincipal, real, primitive, odd character modulo $k$, where $k \ge 7$ is odd, and $a$ is an odd positive integer such that $a-3 < 2k$. Then
\[
  \sum_{ j=1 }^{ k-1 } \chi(j) \cot \left( \frac{ \pi j }{ k } \right) \cos^{ a-1 } \left( \frac{ \pi j }{ k } \right) = 2 \sqrt k \left( h(-k) - 2^{ 1-a } \!\!\!\!\!\!\!\! \sum_{ { n,m,s \ge 0 } \atop { n+2m+s = (a-1)/2 } } \!\!\!\!\!\!\!\! \chi(n) \binom{ a+1 }{ s } \right) .
\]
\end{theorem}

Now we turn to \cite[Theorem 7.1]{berndtzaharescu}.

\begin{theorem}[Berndt--Zaharescu]\label{bzcotcos}
Suppose $\chi$ is a nonprincipal, real, primitive, odd character modulo $k$, where $k \ge 7$ is odd, and $b$ is a positive integer such that $b \le k$. Then
\[
  \sum_{ j=1 }^{ k-1 } \chi(j) \cot \left( \frac{ \pi j }{ k } \right) \cos \left( \frac{ 2 \pi b j }{ k } \right)
  = \sqrt k \left( 2 h(-k) - \chi(b) - 2 \sum_{ n=1 }^{ b-1 } \chi(n) \right) .
\]
\end{theorem}

\begin{proof}[Proof of Theorem~\ref{bzcotcos}]
By applying the Convolution Theorem to the functions
$f(n) = \frac{ i }{ 2k } \sum_{ j=1 }^{ k-1 } \cot \frac{ \pi j }{ k } \omega^{ jn }$
and
$g(n) = \sum_{ j=1 }^{ k-1 } \chi(j) \omega^{ jn }$,
one obtains after a somewhat tedious but straightforward calculation the Fourier transform
\begin{equation}\label{chicottransform}
  F(n) 
  = \frac 1 k \sum_{ j=1 }^{ k-1 } \chi(j) \cot \frac{ \pi j }{ k } \omega^{ jn }
  = \frac{ 1 }{ \sqrt k } \left( 2 h(-k) + \chi(n) - 2 \sum_{ m=0 }^n \chi(m) \right) .
\end{equation}
Now we convolve this function with
\[
  h(n) =
  \begin{cases}
    \frac 1 2 & \text{ if } k | b+n \text{ or } k | b-n , \\
    0 & \text{ otherwise, }
  \end{cases}
\]
the Fourier transform of $\cos \left( \frac{ 2 \pi b n }{ k } \right)$.
The Convolution Theorem gives
\begin{align*}
  &\frac 1 k \sum_{ j=1 }^{ k-1 } \chi(j) \cot \left( \frac{ \pi j }{ k } \right) \cos \left( \frac{ 2 \pi b j }{ k } \right)
   = \sum_{ n=0 }^{ k-1 } F(n) \, h(-n) \\
  &\qquad = \frac{ 1 }{ 2 \sqrt k } \left( 4 h(-k) + \chi(b) + \chi(k-b) - 2 \left( \sum_{ m=0 }^b \chi(m) + \sum_{ m=0 }^{ k-b } \chi(m) \right) \right) \\
  &\qquad = \frac{ 1 }{ \sqrt k } \left( 2 h(-k) - \chi(b) - 2 \sum_{ n=1 }^{ b-1 } \chi(n) \right) . \qedhere
\end{align*}
\end{proof}

There are innumerable siblings of Theorem~\ref{bzcotcos} that one can prove with the same methods, and we will give some of them below. The only difficulty lies in the computation of (more and more involved) Fourier transforms. For example, the following theorem is as easily proved as the previous once one has computed the Fourier transform
\[
  \sum_{ j=1 }^{ k-1 } \chi(j) \cot^2 \left( \frac{ \pi j }{ k } \right) \omega^{ jn } 
  = i \sqrt k \left( 4n \, h(-k) - \chi(n) - 4 \sum_{ m=0 }^n \chi(m) (m-n) \right)
\]
(where $\chi$ satisfies the conditions stated in the theorem).

\begin{theorem}\label{cotsin}
Suppose $\chi$ is a nonprincipal, real, primitive, odd character modulo $k$, where $k \ge 7$ is odd. Then
\[
  \sum_{ j=1 }^{ k-1 } \chi(j) \cot^2 \left( \frac{ \pi j }{ k } \right) \sin \left( \frac{ 2 \pi a j }{ k } \right)
  = \sqrt k \left( 4 a \, h(-k) - \chi(a) - 4 \sum_{ m=0 }^{ a-1 } \chi(m) (a-m) \right) .
\]
\end{theorem}

Our final goal is a discrete Fourier proof of Berndt--Zaharescu's Theorem~\ref{finalthm}, namely, for $\chi$ a nonprincipal, real, primitive, odd character modulo $k$, where $k \ge 7$ is odd, we have
\[
  \sum_{ j=1 }^{ k-1 } \chi(j) \frac{ \sin^2 \left( \frac{ \pi j }{ k } \right) }{ \sin \left( \frac{ 4 \pi j }{ k } \right)  }
  = \frac{ 3 \sqrt k }{ 2 } \left( \chi(2) - 1 \right) h(-k)
  = \begin{cases}
      - 3 \sqrt k \, h(-k) & \text{ if } k \equiv 3 \bmod 8 , \\
      0 & \text{ if } k \equiv 7 \bmod 8 .
    \end{cases}
\]
Since $\frac{ 1 }{ \sin (2 \theta) } = \frac 1 2 \left( \cot \theta + \tan \theta \right)$
and $\cot (2 \theta) = \frac 1 2 \left( \cot \theta - \tan \theta \right)$, 
\begin{align*}
  \sum_{ j=1 }^{ k-1 } \chi(j) \frac{ \sin^2 \left( \frac{ \pi j }{ k } \right) }{ \sin \left( \frac{ 4 \pi j }{ k } \right)  }
  &= \frac 1 4 \sum_{ j=1 }^{ k-1 } \chi(j) \sin^2 \left( \frac{ \pi j }{ k } \right) \cot \left( \frac{ \pi j }{ k } \right)
    - \frac 1 4 \sum_{ j=1 }^{ k-1 } \chi(j) \sin^2 \left( \frac{ \pi j }{ k } \right) \tan \left( \frac{ \pi j }{ k } \right) \\
  &\qquad + \frac 1 2 \sum_{ j=1 }^{ k-1 } \chi(j) \sin^2 \left( \frac{ \pi j }{ k } \right) \tan \left( \frac{ 2 \pi j }{ k } \right) ,
\end{align*}
and so Theorem~\ref{finalthm} follows from Theorems \ref{chisin2cot}, \ref{chisin2tan}, and \ref{chisin2tan2} below, which we believe to be novel.
Since the proofs are very similar to the ones we have given above, we outline only which functions to convolve in each case.
We need the discrete Fourier transform of $\chi(n) \cot \frac{ \pi n }{ k }$ which we derived in \eqref{chicottransform}, and the following transforms which are easily verified.
\begin{align}
  \frac 1 k \sum_{ j=0 }^{ k-1 } \sin^2 \left( \frac{ \pi j }{ k } \right) \omega^{ jn }
  &= \begin{cases}
       \frac 1 2 & \text{ if } k|n, \\
       - \frac 1 4 & \text{ if } k|(n+1) \text{ or } k|(n-1), \\
       0 & \text{ otherwise, }
     \end{cases} \label{sin2transform} \\
  \frac i k \sum_{ j=0 }^{ k-1 } \tan \frac{ 2 \pi j }{ k } \, \omega^{ jn }
  &= \begin{cases}
       0 & \text{ if } k|n, \\
       1 & \text{ if } n \equiv 0, 1 \bmod k \text{ and } k \nmid n, \\
       -1 & \text{ if } n \equiv 2, 3 \bmod k \text{ and } k \nmid n. \\
     \end{cases} \label{tan2transform}
\end{align}
The last identity is only valid for $k \equiv 3 \bmod 4$, but this suffices for our purposes: for a character modulo $k$ to be nonprincipal, real, primitive, and odd, $k$ has to be congruent to $3 \bmod 4$.

\begin{theorem}\label{chisin2cot}
Suppose $\chi$ is a nonprincipal, real, primitive, odd character modulo $k$, where $k \ge 7$ is odd. Then
\[
  \sum_{ j=1 }^{ k-1 } \chi(j) \sin^2 \left( \frac{ \pi j }{ k } \right) \cot \left( \frac{ \pi j }{ k } \right)
  = \frac 1 2 \sqrt k \, .
\]
\end{theorem}

\begin{proof}
Convolve the discrete Fourier transform of $\sin^2 \left( \frac{ \pi n }{ k } \right)$ with that of $\chi(n) \cot \left( \frac{ \pi n }{ k } \right)$.
\end{proof}

\begin{theorem}\label{chisin2tan}
Suppose $\chi$ is a nonprincipal, real, primitive, odd character modulo $k$, where $k \ge 7$ is odd. Then
\begin{align*}
  \sum_{ j=1 }^{ k-1 } \chi(j) \sin^2 \left( \frac{ \pi j }{ k } \right) \tan \left( \frac{ \pi j }{ k } \right)
  &= \sqrt k \left( - \frac 1 2 + \left( 2 - 4 \, \chi(2) \right) h(-k) \right) \\
  &= \begin{cases}
      \sqrt k \left( - \frac 1 2 + 6 h(-k) \right) & \text{ if } k \equiv 3 \bmod 8 , \\
      \sqrt k \left( - \frac 1 2 - 2 h(-k) \right) & \text{ if } k \equiv 7 \bmod 8 .
    \end{cases}
\end{align*}
\end{theorem}

\begin{proof}
Let $t(n) := \frac 1 k \sum_{ j=0 }^{ k-1 } \chi(j) \tan \frac{ \pi j }{ k } \, \omega^{ jn }$; we will convolve this function with the discrete Fourier transform of $\sin^2 \left( \frac{ \pi n }{ k } \right)$. Because of the special form of the latter, we only need to compute
\[
  t(0) = - \frac{ 1 }{ \sqrt k } \left( 4 \, \chi(2) - 2 \right) h(-k)
  \qquad \text{ and } \qquad
  t(1) = \frac{ 1 }{ \sqrt k } \left( 1 + \left( 4 \, \chi(2) - 2 \right) h(-k) \right) ,
\]
and the statement now follows with the Convolution Theorem.
\end{proof}

\begin{theorem}\label{chisin2tan2}
Suppose $\chi$ is a nonprincipal, real, primitive, odd character modulo $k$, where $k \ge 7$ is odd. Then
\begin{align*}
  \sum_{ j=1 }^{ k-1 } \chi(j) \sin^2 \left( \frac{ \pi j }{ k } \right) \tan \left( \frac{ 2 \pi j }{ k } \right)
  &= \sqrt k \left( - \frac 1 2 + \left( \chi(2) - 2 \right) h(-k) \right) \\
  &= \begin{cases}
      \sqrt k \left( - \frac 1 2 - 3 \, h(-k) \right) & \text{ if } k \equiv 3 \bmod 8 , \\
      \sqrt k \left( - \frac 1 2 - h(-k) \right) & \text{ if } k \equiv 7 \bmod 8 .
    \end{cases}
\end{align*}
\end{theorem}

\begin{proof}
We convolve the Fourier transform of $\sin^2 \left( \frac{ \pi n }{ k } \right)$ with $s(n) := \frac 1 k \sum_{ j=1 }^{ k-1 } \chi(j) \tan \left( \frac{ 2 \pi j }{ k } \right)$.
Again we only need to know only a few of the values of $s(n)$, since by \eqref{sin2transform} we obtain
\[
  \frac 1 k \sum_{ j=1 }^{ k-1 } \chi(j) \sin^2 \left( \frac{ \pi j }{ k } \right) \tan \left( \frac{ 2 \pi j }{ k } \right)
  = \frac 1 2 \, s(0) - \frac 1 4 \left( s(1) + s(k-1) \right) .
\]
The value $s(0)$ can be easily derived by a quick convolution calculation:
\[
  s(0) = 
    \frac 1 k \sum_{ j=1 }^{ k-1 } \chi(j) \tan \left( \frac{ 2 \pi j }{ k } \right)
  = \frac 1 {\sqrt k} \left( 2 \, \chi(2) - 4 \right) h(-k)
  = \begin{cases}
      - \frac{ 6 }{ \sqrt k } \, h(-k) & \text{ if } k \equiv 3 \bmod 8 , \\
      - \frac{ 2 }{ \sqrt k } \, h(-k) & \text{ if } k \equiv 7 \bmod 8 .
    \end{cases}
\]
It remains to compute $s(1) + s(k-1)$. Let's denote the right-hand side of \eqref{tan2transform} by
\[
  r(n) := \begin{cases}
       0 & \text{ if } k|n, \\
       1 & \text{ if } n \equiv 0, 1 \bmod k \text{ and } k \nmid n, \\
       -1 & \text{ if } n \equiv 2, 3 \bmod k \text{ and } k \nmid n. \\
     \end{cases}
\]
Thus by the Convolution Theorem, $s(n) = \frac{ 1 }{ \sqrt k } \sum_{ m=0 }^{ k-1 } r(n-m) \, \chi(m)$, whence
\[
  s(1) + s(k-1) = \frac{ 1 }{ \sqrt k } \left( \sum_{ m=0 }^{ k-1 } r(1-m) \, \chi(m) + \sum_{ m=0 }^{ k-1 } r(k-1-m) \, \chi(m) \right) .
\]
Since $\chi$ and $r$ are both odd and $\chi(0) = r(0) = 0$,
\begin{align*}
  s(1) + s(k-1)
  &= - \frac{ 1 }{ \sqrt k } \left( \sum_{ m=2 }^{ k-1 } r(m-1) \, \chi(m) + \sum_{ m=1 }^{ k-2 } r(m+1) \, \chi(m) \right) \\
  &= - \frac{ 1 }{ \sqrt k } \left( r(k-2) \, \chi(k-1) + r(2) \, \chi(1) + \sum_{ m=2 }^{ k-2 } \chi(m) \bigl( r(m-1) + r(m+1) \bigr) \right) \\
  &= \frac{ 2 }{ \sqrt k } \, .
\end{align*}
Here the last equation follows from the fact that for $2 \le m \le k-2$, $r(m-1) + r(m+1) = 0$, as can be directly concluded from the definition of $r$.
\end{proof}




\bibliographystyle{amsplain}

\def\cprime{$'$} \def\cprime{$'$}
\providecommand{\bysame}{\leavevmode\hbox to3em{\hrulefill}\thinspace}
\providecommand{\MR}{\relax\ifhmode\unskip\space\fi MR }
\providecommand{\MRhref}[2]{%
  \href{http://www.ams.org/mathscinet-getitem?mr=#1}{#2}
}
\providecommand{\href}[2]{#2}

\setlength{\parskip}{0cm} 
\end{document}